\documentclass[12pt,thmsa]{article}

\usepackage{amssymb}
\usepackage{amsmath}
\usepackage{amsfonts}
\newtheorem{theorem}{Theorem}[section]
\newtheorem{proposition}[theorem]{Proposition}
\newtheorem{proof}[theorem]{Proof}
\newtheorem{corollary}[theorem]{Corollary}
\newtheorem{definition}[theorem]{Definition}
\newtheorem{remark}[theorem]{Remark}

\begin{document}

\title{\textbf{ON} \textbf{KNOTTED SPHERES IN EUCLIDEAN 4-SPACE }$\mathbb{E}%
^{4}$}
\author{Kadri Arslan }
\maketitle

\begin{abstract}
In the present study we consider knotted spheres in Euclidean $4$-space $%
\mathbb{E}^{4}$. Firstly, we give some basic curvature properties of knotted
spheres in $\mathbb{E}^{4}$. \ Further, we obtained some results related
with the conjugate nets and Laplace transforms of these kind of surfaces.
\end{abstract}

\section{\textbf{Introduction}}

\footnote{%
2010 AMS \textit{Mathematics Subject Classification}. 53C40, 53C42
\par
\textit{Key words and phrases}: Rotational surface, Knotted sphere,
Conjugate nets}Let us consider a unit speed regular curve $\gamma $ in $%
\mathbb{E}^{4}$ and a unit speed spherical curve $\rho $ in $\mathbb{E}^{2}.$
Then, the rotation of $\gamma $ around $\rho $ give rise a surface $M$ in $%
\mathbb{E}^{4}$, which is called \textit{rotational surface}. The rotational
surfaces in $\mathbb{E}^{4}$ was first introduced by C. Moore in $1919$
(see, \cite{Mo}). Further, many researchers concentrated these studies on
this subject, see for example \cite{ABCO}, \cite{BABO1}, \cite{DT} and \cite%
{GM1}. The rotational surfaces in $\mathbb{E}^{4}$ with constant curvatures
are studied in \cite{Wo}.

Let us denote the half-space $x_{3}\geq 0,x_{4}=0$ by $\mathbb{E}_{+}^{3}(0)$
and take an arc $\alpha $ with the end point in the plane $x_{3}=0,x_{4}=0$ $%
($denoted by $\Pi ).$ The rotation the half space plane $\mathbb{E}%
_{+}^{3}(0)$ by the angle $v$ around the plane $\Pi $ is denoted by $\mathbb{%
E}_{+}^{3}(v).$ Consequently, after the rotation the point with coordinates $%
x_{1},x_{2},x_{3},x_{4}$ passes into the point with the coordinates $%
\widetilde{x}_{1},\widetilde{x}_{2},\widetilde{x}_{3},\widetilde{x}_{4}$ by%
\begin{eqnarray}
\widetilde{x}_{1} &=&x_{1}  \notag \\
\widetilde{x}_{2} &=&x_{2}  \label{z1} \\
\widetilde{x}_{3} &=&x_{3}\cos v-x_{4}\sin v  \notag \\
\widetilde{x}_{3} &=&x_{3}\sin v-x_{4}\cos v.  \notag
\end{eqnarray}

In rotation by $360^{\circ }$ the points of $\alpha $ being in $\mathbb{E}%
_{+}^{3}(v),$ form the set $M$ homeomorphic to $S^{2}$ \cite{Ar}$.$ Let $%
\alpha $ be a smooth curve with tangent vectors at $p$ and\ $q$ orthogonal
to $\Pi .$ Then $M$ is a smooth surface which is called \textit{knotted
sphere} in $\mathbb{E}^{4}$ \cite{Am}.

A net of curves on a surface $M$ is called \textit{conjugate}, if at every
point the tangent directions of the curves of the net separate harmonically
the asymptotic directions \cite{LG}. Consequently, for a surface $M$ with a
patch $X(u,v)$, a net of curves on $M$ are conjugate if and only if the
second partial derivative $X_{uv}$ of $X$ lies in the subspace spanned by $%
X_{u}$ and $X_{v}$ \cite{KT}.

This paper is organized as follows: In section $2$ we give some basic
concepts of the second fundamental form and curvatures of the surfaces in $%
\mathbb{E}^{4}$. In Section $3$ we consider knotted spheres in $\mathbb{E}%
^{4}$. Firstly, we give some basic curvature properties of knotted spheres
in $\mathbb{E}^{4}$. \ Further, we introduce some kind of knotted spheres
and obtained some results related with their curvatures. In section $4$ we
give some basic curvature properties of the conjugate nets on a surface in $%
E^{n}.$ Further, we calculated the Laplace invariants and the Laplace
transforms of the knotted sphere in $\mathbb{E}^{4}.$

\section{\textbf{Preliminaries}}

Let $M$ be a local surface in $\mathbb{E}^{n}$ given with position vector $%
X(u,v)$. The tangent space $T_{p}M$ is spanned by the vector fields $X_{u}$
and $X_{v}$. In the chart $(u,v)$ the coefficients of the first fundamental
form of $M$ are given by 
\begin{equation}
E=\langle X_{u},X_{u}\rangle ,F=\left \langle X_{u},X_{v}\right \rangle
,G=\left \langle X_{v},X_{v}\right \rangle ,  \label{a1}
\end{equation}%
where $\left \langle ,\right \rangle $ is the inner product in $\mathbb{E}%
^{n}$. We assume that $X(u,v)$ is regular i.e., $W^{2}=EG-F^{2}\neq 0$ \cite%
{G}.

Consequently, the \textit{Gaussian curvature} of $M$ is given by%
\begin{equation}
K=-\frac{1}{4W^{2}}\left \vert 
\begin{array}{ccc}
E & E_{u} & E_{v} \\ 
F & F_{u} & F_{v} \\ 
G & G_{u} & G_{v}%
\end{array}%
\right \vert -\frac{1}{2W}\left( \left( \frac{E_{v}-F_{u}}{W}\right)
_{v}-\left( \frac{F_{v}-G_{u}}{W}\right) _{u}\right) .  \label{a2}
\end{equation}

Let $\overset{\sim }{\nabla }$ be the Riemannian connection of $\mathbb{E}%
^{4},$ and $X_{1}=X_{u},$ $X_{2}=X_{v}$ tangent vector fields of $M$ then 
\textit{Gauss equation} gives%
\begin{equation}
\widetilde{\nabla }_{X_{i}}X_{j}=\sum_{k=1}^{2}\Gamma _{ij}^{k}X_{k}+h\left(
X_{i},X_{j}\right) ;\text{ }1\leq i,j\leq 2,  \label{a3}
\end{equation}%
where $h$ is the second fundamental form and $\Gamma _{ij}^{k}$ are the 
\textit{Christoffel symbols} of $M$.

The \textit{mean curvature vector} $\overrightarrow{H}$ of of $M$ is given
by 
\begin{equation}
\overrightarrow{H}=\frac{1}{2W^{2}}\left(
Eh(X_{v},X_{v})-2Fh(X_{u},X_{v})+Gh(X_{u},X_{u})\right) .  \label{a4}
\end{equation}

The norm of the mean curvature vector $\overrightarrow{H}$ is known as 
\textit{mean curvature} of $M$. Recall that, a surface $M$ is said to be 
\textit{minimal} (resp. \textit{flat}) if its mean curvature (resp. Gaussian
curvature) vanishes identically \cite{Ch}.

\section{\textbf{Knotted Spheres in }$\mathbb{E}^{4}$}

Let $M$ be a knotted sphere\ given with (\ref{z1}), then the position vector
of $M$ can be represented as: 
\begin{equation}
X(u,v)=\left( x_{1}(u),x_{2}(u),x_{3}(u)\cos v-x_{4}(u)\sin v,x_{3}(u)\sin
v+x_{4}(u)\cos v\right) .  \label{b2}
\end{equation}%
where 
\begin{equation*}
\gamma (u)=\left( x_{1}(u),x_{2}(u),x_{3}(u),x_{4}(u)\right) ,
\end{equation*}%
is the profile curve of the surface \cite{Am}. Then, the tangent space $%
T_{p}M$ of $M$ is spanned by

\begin{eqnarray}
X_{u} &=&\left( x_{1}^{\prime }(u),x_{2}^{\prime }(u),x_{3}^{\prime }(u)\cos
v-x_{4}^{\prime }(u)\sin v,x_{3}^{\prime }(u)\sin v+x_{4}^{\prime }(u)\cos
v\right) ,  \notag \\
X_{v} &=&\left( 0,0,-x_{3}(u)\sin v-x_{4}(u)\cos v,x_{3}(u)\cos
v-x_{4}(u)\sin v\right) .  \label{b3}
\end{eqnarray}

Consequently, the coefficients of first fundamental form become%
\begin{eqnarray}
E &=&\left \langle X_{u},X_{u}\right \rangle =1,  \notag \\
F &=&\left \langle X_{u},X_{v}\right \rangle =x_{3}(u)x_{4}^{\prime
}(u)-x_{3}^{\prime }(u)x_{4}(u),\text{ }  \label{b4} \\
G &=&\left \langle X_{v},X_{v}\right \rangle =x_{3}^{2}(u)+x_{4}^{2}(u). 
\notag
\end{eqnarray}

The Christoffel symbols $\Gamma _{ij}^{k}$ of the canal surface $M$ are
given by%
\begin{equation}
\begin{array}{ccc}
\text{ \  \ }\Gamma _{11}^{1}=-\frac{FF_{u}}{W^{2}}, & \text{ \ }\Gamma
_{12}^{1}=-\frac{FG_{u}}{2W^{2}}, & \text{ \ }\Gamma _{22}^{1}=-\frac{GG_{u}%
}{2W^{2}} \\ 
\Gamma _{11}^{2}=\frac{F_{u}}{W^{2}}, & \Gamma _{12}^{2}=\frac{G_{u}}{2W^{2}}%
, & \text{\ }\Gamma _{22}^{2}=\frac{FG_{u}}{2W^{2}}.%
\end{array}
\label{b5}
\end{equation}%
which are symmetric with respect to the covariant indices $($\cite{G}, $%
p.398)$.

The second partial derivatives of \ $r$ are expressed as follows:%
\begin{eqnarray}
X_{uu} &=&\left( x_{1}^{\prime \prime }(u),x_{2}^{\prime \prime
}(u),x_{3}^{\prime \prime }(u)\cos v-x_{4}^{\prime \prime }(u)\sin
v,x_{3}^{\prime \prime }(u)\sin v+x_{4}^{\prime \prime }(u)\cos v\right) , 
\notag \\
X_{uv} &=&\left( 0,0,-x_{3}^{\prime }(u)\sin v-x_{4}^{\prime }(u)\cos
v,x_{3}^{\prime }(u)\cos v-x_{4}^{\prime }(u)\sin v\right) ,  \label{b6} \\
X_{vv} &=&\left( 0,0,-x_{3}(u)\cos v+x_{4}(u)\sin v,-x_{3}(u)\sin
v-x_{4}(u)\cos v\right) .  \notag
\end{eqnarray}

Using (\ref{b3}) with (\ref{b6}) we get 
\begin{eqnarray}
\left \langle X_{uu},X_{vv}\right \rangle &=&-\left( x_{3}(u)x_{3}^{\prime
\prime }(u)+x_{4}(u)x_{4}^{\prime \prime }(u)\right) ,  \notag \\
\left \langle X_{vv},X_{u}\right \rangle &=&-\left( x_{3}(u)x_{3}^{\prime
}(u)+x_{4}(u)x_{4}^{\prime }(u)\right) ,  \notag \\
\left \langle X_{uu},X_{v}\right \rangle &=&x_{3}(u)x_{4}^{\prime \prime
}(u)-x_{3}^{\prime \prime }(u)x_{4}(u),  \label{b7*} \\
\left \langle r_{s\theta },X_{uv}\right \rangle &=&(x_{3}^{\prime
}(u))^{2}+(x_{4}^{\prime }(u))^{2},  \notag \\
\left \langle X_{uu},X_{u}\right \rangle &=&0,  \notag \\
\left \langle X_{vv},X_{v}\right \rangle &=&0.  \notag
\end{eqnarray}

Hence, taking into account (\ref{a3}), the Gauss equation implies the
following equations; 
\begin{eqnarray}
\widetilde{\nabla }_{X_{u}}X_{u} &=&X_{uu}=\Gamma _{11}^{1}X_{u}+\Gamma
_{11}^{2}X_{v}+h(X_{u},X_{u}),  \label{b6*} \\
\widetilde{\nabla }_{X_{u}}X_{v} &=&X_{uv}=\Gamma _{12}^{1}X_{u}+\Gamma
_{12}^{2}X_{v}+h(X_{u},X_{v}),  \notag \\
\widetilde{\nabla }_{X_{v}}X_{v} &=&X_{vv}=\Gamma _{22}^{1}X_{u}+\Gamma
_{22}^{2}X_{v}+h(X_{v},X_{v}).  \notag
\end{eqnarray}

Taking in mind (\ref{a3}), (\ref{b5}) and (\ref{b6*}) we get%
\begin{eqnarray}
h(X_{u},X_{u}) &=&X_{uu}+\frac{FF_{u}}{W^{2}}X_{u}-\frac{F_{u}}{W^{2}}X_{v},
\notag \\
h(X_{u},X_{v}) &=&X_{uv}+\frac{FG_{u}}{2W^{2}}X_{u}-\frac{G_{u}}{2W^{2}}%
X_{v},  \label{b7} \\
h(X_{v},X_{v}) &=&X_{vv}+\frac{GG_{u}}{2W^{2}}X_{u}-\frac{FG_{u}}{2W^{2}}%
X_{v}.  \notag
\end{eqnarray}

Consequently, by the use of \ (\ref{a2}), (\ref{a4}), (\ref{b4}) and (\ref%
{b7*}) with (\ref{b7}) the Gaussian curvature and mean curvature vector of $%
M $ become%
\begin{equation}
K=-\frac{1}{2W}\left( \frac{G_{u}}{W}\right) _{u},  \label{b8}
\end{equation}%
and%
\begin{equation}
\overrightarrow{H}=\frac{1}{2W^{2}}\left(
Eh(X_{v},X_{v})-2Fh(X_{u},X_{v})+Gh(X_{u},X_{u})\right) ,  \label{b9}
\end{equation}%
respectively.

In the sequel, we consider some special cases;

\textbf{Case I. }Suppose 
\begin{equation}
x_{3}=\cos \varphi (u),\text{ }x_{4}=\sin \varphi (u),  \label{b10}
\end{equation}%
then the position vector of the knotted sphere $M$ can be represented as

\begin{equation}
X(u,v)=\left( x_{1},x_{2},\cos \varphi (u)\cos v-\sin \varphi (u)\sin v,\cos
\varphi (u)\sin v+\sin \varphi (u)\cos v\right) .  \label{b10*}
\end{equation}

Hence, the coefficients of the first fundamental form become%
\begin{eqnarray}
E &=&\left \langle X_{u},X_{u}\right \rangle =1,  \notag \\
F &=&\left \langle X_{u},X_{v}\right \rangle =\varphi ^{\prime }(u),\text{ }
\label{b11} \\
G &=&\left \langle X_{v},X_{v}\right \rangle =1,  \notag
\end{eqnarray}%
where $\varphi $ is a differentiable (angle) function.

Summing up (\ref{b5})-(\ref{b11}) the following results are proved;

\begin{proposition}
The surface $M$ given with the position vector (\ref{b10*}) is a flat
surface.
\end{proposition}

\begin{proposition}
Let $M$ be a surface given with the position vector (\ref{b10*}). Then, the
mean curvature $\left \Vert \overrightarrow{H}\right \Vert $ of $M$ at point 
$p $ is given by%
\begin{equation}
\left \Vert \overrightarrow{H}\right \Vert =\frac{1}{4\left( 1-\left(
\varphi ^{\prime }(u)\right) ^{2}\right) ^{2}}\left( \kappa _{\gamma
}^{2}+1-2\left( \varphi ^{\prime }(u)\right) ^{2}-\frac{\left( \varphi
^{\prime \prime }(u)\right) ^{2}}{1-\left( \varphi ^{\prime }(u)\right) ^{2}}%
\right) ,  \label{b12}
\end{equation}%
where $\kappa _{\gamma }$ is the curvature of the profile curve $\gamma .$
\end{proposition}

\begin{proof}
With the help of (\ref{b7}), (\ref{b9}) and (\ref{b11}) the Gaussian
curvature vector of $M$ becomes%
\begin{equation}
2\overrightarrow{H}=\frac{1}{1-\left( \varphi ^{\prime }\right) ^{2}}\left( 
\overline{x}_{1},\overline{x}_{2},\overline{x}_{3}\cos v+\overline{x}%
_{4}\sin v,\overline{x}_{3}\sin v-\overline{x}_{4}\cos v\right) ,
\label{b13}
\end{equation}%
where,%
\begin{eqnarray}
\overline{x}_{1} &=&x_{1}^{\prime \prime }+\frac{\varphi ^{\prime }\varphi
^{\prime \prime }}{1-\left( \varphi ^{\prime }\right) ^{2}}x_{1}^{\prime }, 
\notag \\
\overline{x}_{2} &=&x_{1}^{\prime \prime }+\frac{\varphi ^{\prime }\varphi
^{\prime \prime }}{1-\left( \varphi ^{\prime }\right) ^{2}}x_{2}^{\prime }
\label{b14} \\
\overline{x}_{3} &=&\frac{\varphi ^{\prime }\varphi ^{\prime \prime
}x_{3}^{\prime }+\varphi ^{\prime \prime }x_{4}}{1-\left( \varphi ^{\prime
}\right) ^{2}}+x_{3}^{\prime \prime }-x_{3}+2\varphi ^{\prime }x_{4}^{\prime
}=-\left( 1-\left( \varphi ^{\prime }\right) ^{2}\right) \cos \varphi , 
\notag \\
\overline{x}_{4} &=&\frac{\varphi ^{\prime \prime }x_{3}-\varphi ^{\prime
}\varphi ^{\prime \prime }x_{4}^{\prime }}{1-\left( \varphi ^{\prime
}\right) ^{2}}-x_{4}^{\prime \prime }+x_{4}+2\varphi ^{\prime }x_{3}^{\prime
}=\left( 1-\left( \varphi ^{\prime }\right) ^{2}\right) \sin \varphi  \notag
\end{eqnarray}%
are differentiable functions. Taking the norm of the vector (\ref{b13}) and
using (\ref{b10}) with (\ref{b14}) we obtain (\ref{b12}). This completes the
proof of the proposition.
\end{proof}

As a consequence of Proposition $2$ we obtain the following result.

\begin{corollary}
Let $M$ be a surface given with the position vector (\ref{b10*}). Then $M$
is a minimal surface if and only if \ the curvature $\kappa _{\gamma }$ of
the profile curve $\gamma $ satisfies the equality 
\begin{equation}
\kappa _{\gamma }^{2}=\frac{\left( \varphi ^{\prime \prime }(u)\right) ^{2}}{%
1-\left( \varphi ^{\prime }(u)\right) ^{2}}+2\left( \varphi ^{\prime
}(u)\right) ^{2}-1,  \label{b15}
\end{equation}%
in such a way that the (angle) function $\varphi $ is non-constant.
\end{corollary}

\textbf{Case II. }Suppose $x_{4}=\lambda x_{3}$, $\lambda \in \mathbb{R},$
then the position vector of the knotted sphere $M$ can be represented as 
\begin{equation}
r(s,\theta )=\left( x_{1}(u),x_{2}(u),x_{3}(u)\left( \cos v-\lambda \sin
v\right) ,x_{3}(u)\left( \sin v+\lambda \cos v\right) \right) .  \label{b16}
\end{equation}

Hence, the coefficients of the first fundamental form of $M$ become%
\begin{eqnarray}
E &=&\left \langle X_{u},X_{u}\right \rangle =1,  \notag \\
F &=&\left \langle X_{u},X_{v}\right \rangle =0,\text{ }  \label{b16*} \\
G &=&\left \langle X_{v},X_{v}\right \rangle =\left( 1+\lambda ^{2}\right)
x_{3}^{2}(u).  \notag
\end{eqnarray}

By the use of (\ref{b16*}) with (\ref{a2}) we obtain the following result.

\begin{proposition}
Let $M$ be a surface given with the position vector (\ref{b16}). Then, the
Gaussian curvature of $M$ is given by%
\begin{equation*}
K=-\frac{x_{3}^{\prime \prime }(u)}{x_{3}(u)}.
\end{equation*}
\end{proposition}

As a consequence of Proposition $4$ we obtain the following result.

\begin{corollary}
Let $M$ be a surface given with the position vector (\ref{b16}). Then we
have the following statements

$i)$ If $x_{3}(u)=$ $ae^{cu}+be^{-cu}$ then the corresponding surface is
pseudo-spherical, i.e., it has negative Gaussian curvature $K=-\frac{1}{c^{2}%
},$

$ii)$ If $x_{3}(u)=$ $a\cos cu+b\sin cu$ then the corresponding surface is
spherical, i.e., it has negative Gaussian curvature $K=\frac{1}{c^{2}},$
where $a$, $b$ and $c$ are real constants.

$iii)$ If $x_{3}(u)=$ $au+b$ then the corresponding surface is flat.
\end{corollary}

\section{\textbf{The Conjugate Nets and Laplace Transforms of Knotted Sphere}%
}

In the present section, we will give some basic relations of the conjugate
net of curves on a surface in $\mathbb{R}^{n}.$ A net of curves on a \
surface $M$ is called conjugate, if at every point the tangent directions of
the curves of the net separate harmonically the asymptotic directions.
Taking the net to be parametric net with parameters $u$ and $v$, the
classical notion of the conjugate net usually can be stated in \cite{LG} as
follows:

\begin{definition}
Let $M$ be a smooth surface given with the position vector $X:U\subset 
\mathbb{E}^{2}\rightarrow \mathbb{E}^{n}$, and $N_{1},...,$ $N_{n-2}$ normal
vector fields of $M$ in $\mathbb{E}^{n}.$ If $X_{uv}=\frac{\partial ^{2}X}{%
\partial u\partial v}$ satisfies 
\begin{equation}
\left \langle X_{uv},N_{\alpha }\right \rangle =0,1\leq \alpha \leq n-2,
\label{c1}
\end{equation}%
then $\left( u,v\right) $ is called conjugate coordinates of $X$ and the net
woven by coordinate curves is called the conjugate net. For convenience, we
denote the conjugate net by $\left( u,v\right) $. Here, $\left \langle
,\right \rangle $ denotes the inner product on $\mathbb{E}^{n}$ \cite{LG}$.$
\end{definition}

Equation (\ref{c1}) is equivalent to the condition that $X_{uv}$ lies in the
subspace spanned by $X_{u}$ and $X_{v}$; i.e.,%
\begin{equation}
X_{uv}=\Gamma _{12}^{1}X_{u}+\Gamma _{12}^{2}X_{v}.  \label{c2*}
\end{equation}

Now, for the surface with normal conjugate net, we have two transforms%
\begin{equation}
X_{1}=X-\frac{X_{v}}{\Gamma _{12}^{1}},\text{ }X_{-1}=X-\frac{X_{u}}{\Gamma
_{12}^{2}},  \label{c2**}
\end{equation}%
which are called the \textit{Laplace transforms} of surface $M$ \cite{KT}.

Furthermore, the functions%
\begin{equation}
h=\partial _{u}\Gamma _{12}^{1}-\Gamma _{12}^{1}\Gamma _{12}^{2},\text{ }%
k=\partial _{v}\Gamma _{12}^{2}-\Gamma _{12}^{1}\Gamma _{12}^{2}  \label{c2}
\end{equation}%
are called the \textit{Laplace invariants}.

If $\Gamma _{12}^{1}\neq 0$ $(resp.$ $\Gamma _{12}^{2}\neq 0),$ the
conjugate net is called $v-$\textit{direction normal} (resp. $u-$\textit{%
direction normal}). To establish geometrically the notion of conjugate net
in ambient space, the following result explain the real geometric meaning of
the conjugate net defined in (\ref{c1}).

\begin{proposition}
\cite{LG} $\left( u,v\right) $ is a $v$-direction normal conjugate net of $M$
if and only if there exists another surface $\widetilde{M}$ given with the
position vector $X_{1}(u,v)$ such that, for any $\left( u,v\right) \in
D\subset \mathbb{E}^{2},$ the straight line $XX_{1}$ joining the points $%
X(u,v)$ and $X_{1}(u,v)$ is parallel to the vectors $X_{v}(u,v)$ and $%
X_{1u}(u,v).$
\end{proposition}

We obtain the following results,

\begin{theorem}
Let $M$ be a surface given with the position vector (\ref{b2}). If $\left(
u,v\right) $ are conjugate coordinates, then $M$ is a flat surface.
\end{theorem}

\begin{proof}
Let $\left( u,v\right) $ be conjugate coordinates of the knotted sphere $M$
given with the parametrization (\ref{b2}). Then, by definition $%
h(X_{u},X_{v})=0.$ So, by the use of (\ref{b7}) we have 
\begin{equation}
X_{uv}=\frac{G_{u}}{2W^{2}}X_{v}-\frac{FG_{u}}{2W^{2}}X_{u}.  \label{c3}
\end{equation}%
Consequently, substituting (\ref{b3}) with (\ref{b6}) into (\ref{c3}) we get 
\begin{eqnarray}
x_{1}^{\prime }(u) &=&0,  \notag \\
x_{2}^{\prime }(u) &=&0,  \notag \\
x_{3}^{\prime }(u) &=&x_{3}(u)\frac{G_{u}}{2W^{2}}-x_{4}^{\prime }(u)\frac{%
FG_{u}}{2W^{2}}  \label{c4} \\
x_{4}^{\prime }(u) &=&x_{4}(u)\frac{G_{u}}{2W^{2}}+x_{3}^{\prime }(u)\frac{%
FG_{u}}{2W^{2}}.  \notag
\end{eqnarray}%
Summing up the last two equations of (\ref{c4}) we obtain 
\begin{equation}
\left( x_{3}^{\prime }(u)\right) ^{2}+\left( x_{4}^{\prime }(u)\right) ^{2}=%
\frac{G_{u}}{2W^{2}}\left( x_{3}(u)x_{3}^{\prime }(u)+x_{4}(u)x_{4}^{\prime
}(u)\right) .  \label{c5}
\end{equation}%
Moreover, the profile curve $\gamma $ has arclength parameter and the first
two equations of (\ref{c4}) imply that 
\begin{equation}
\left( x_{3}^{\prime }(u)\right) ^{2}+\left( x_{4}^{\prime }(u)\right)
^{2}=1.  \label{c6}
\end{equation}%
Hence, by the use of (\ref{b4}) \ with (\ref{c6}) the equation (\ref{c5})
reduces to 
\begin{equation}
1=\frac{G_{u}^{2}}{4W^{2}},\text{ }W>0.  \label{c7}
\end{equation}%
Thus, substituting (\ref{c7}) into (\ref{b8}) we get $K=-\frac{1}{2W}\left( 
\frac{G_{u}}{W}\right) _{u}=0.$This completes the proof of the proposition.
\end{proof}

By the virtue of (\ref{b8}) the following results hold.

\begin{corollary}
The coordinates $(u,v)$ of the surface $M$ given with the position vector (%
\ref{b10*}) can not be conjugate.
\end{corollary}

\begin{proof}
Suppose that $\left( u,v\right) $ are the conjugate coordinates of \ the
surface given with the parametrization (\ref{b10*}). Then, from (\ref{b11})
and (\ref{c7}) we get $4W^{2}=G_{u}^{2}=0.$ But, this contradicts the fact
that $W>0.$ So, the coordinates $(u,v)$ can not be conjugate.
\end{proof}

\begin{remark}
Corollary 7 shows that the inverse statement of Theorem $7$ may not be true.
\end{remark}

\begin{proposition}
Let $M$ be a knotted sphere\ given with the position vector (\ref{b16}).Then
the conjugate surface $\widetilde{M}$ is a part of the rotation plane $\Pi .$
\end{proposition}

\begin{proof}
Let $M$ be a knotted sphere\ given with the parametrization (\ref{b2}), then
by the use of (\ref{b5}) we get \ 
\begin{equation*}
\Gamma _{12}^{1}=-\frac{FG_{u}}{2W^{2}},\Gamma _{12}^{2}=\frac{G_{u}}{2W^{2}}%
.
\end{equation*}%
Now, assume that the surface $M$ with normal conjugate net, then equation (%
\ref{b16*}) yields $\Gamma _{12}^{1}=0$ and $\Gamma _{12}^{2}=\frac{%
x_{3}^{\prime }(s)}{x_{3}(s)}.$ Consequently, the Laplace transform $X_{-1}$
becomes%
\begin{equation*}
X_{-1}=X-\frac{x_{3}(u)}{x_{3}^{\prime }(u)}X_{u}.
\end{equation*}%
Hence, using (\ref{b16}) with its partial derivative $X_{u}$ we obtain 
\begin{equation}
X_{-1}=\left( x_{1}(u)-\frac{x_{3}(u)}{x_{3}^{\prime }(u)}x_{1}^{\prime
}(u),x_{2}(u)-\frac{x_{3}(u)}{x_{3}^{\prime }(u)}x_{2}^{\prime
}(u),0,0\right) .
\end{equation}%
This completes the proof of the proposition.
\end{proof}

\bigskip 
\begin{tabular}{l}
Kadri Arslan \\ 
Department of Mathematics \\ 
Uluda\u{g} University \\ 
16059 Bursa, TURKEY \\ 
E-mail: arslan@uludag.edu.tr%
\end{tabular}

\end{document}